\documentclass{amsart}
\usepackage{amsfonts}
\usepackage{mathrsfs}
\usepackage{amsmath,amsthm,amscd,amsfonts,mathrsfs,amssymb,graphicx,enumerate}
\usepackage{xypic}

\newtheorem{thm}{Theorem}[section]
\newtheorem{cor}[thm]{Corollary}
\newtheorem{lem}[thm]{Lemma}
\newtheorem{prop}[thm]{Proposition}
\theoremstyle{definition}
\newtheorem{defn}[thm]{Definition}

\newtheorem{ques}[thm]{Question}

\newcommand{\PP}{\ensuremath\mathbb{P}}
\newcommand{\calO}{\ensuremath\mathcal{O}}

\begin{document}
\title{Rational curves on Fermat hypersurfaces}
\author{Mingmin Shen}
\address{DPMMS, University of Cambridge, Wilberforce Road, Cambridge CB3 0WB, UK}
\email{M.Shen@dpmms.cam.ac.uk}

\subjclass[2010]{14C05,14D10}

\keywords{rational curve, Fermat hypersurface, positive
characteristic}
\date{}
\maketitle

\begin{abstract}
In this note we study rational curves on degree $p^r+1$ Fermat
hypersurface in $\PP^{p^r+1}_k$, where $k$ is an algebraically
closed field of characteristic $p$. The key point is that the
presence of Frobenius morphism makes the behavior of rational curves
to be very different from that of charateristic 0. We show that if
there exists $N_0$ such that for all $e\geq N_0$ there is a degree
$e$ very free rational curve on $X$, then $N_0> p^r(p^r-1)$.
\end{abstract}

\section{Introduction}
Rational curves appear to be very important in the study of higher
dimensional algebraic varieties. We refer to \cite{kollar} for the
background. Let $X$ be a smooth projective variety over an
algebraically closed field $k$.
\begin{defn}
A rational curve $f:C\cong\PP^1\rightarrow X$ is free (resp. very
free) if $f^*T_X$ is globally generated (resp. ample). We say that
$X$ is \textit{separably rationally connected} (SRC) if there is a
very free rational curve on $X$.
\end{defn}

\begin{defn}
$X$ is \textit{rationally connected} (RC) if a general pair of
points can be connected by a rational curve. This means that there
is a family of rational curves $\pi:U\rightarrow Y$ together with a
morphism $u:U\rightarrow X$ such that the natural map
$u^{(2)}:U\times_Y U\rightarrow X\times_k X$ is dominant. If we only
require the general fiber of $\pi$ to be a genus 0 curve, then we
say that $X$ is rationally chain connected.
\end{defn}
One very important tool to study rational curves is deformation
theory. This works especially well in characteristic 0. For example,
it is easy to see that SRC implies RC in any characteristic. But if
the characteristic is 0, then RC is equivalent to SRC. One very
important class of rationally connected varieties is provided by the
following
\begin{thm}(\cite{kmm},\cite{campana})
Smooth Fano varieties over a field of characteristic 0 are
rationally connected.
\end{thm}

The case of characteristic $p$ is still mystery. We know that all
smooth Fano varieties are rationally chain connected, see V.2 of
\cite{kollar}. Koll\'ar has constructed examples of singular Fano's
that are not SRC, see V.5 of \cite{kollar}. This naturally leads to
the question whether all smooth Fano varieties are SRC. Recently Y.
Zhu has proved that a general Fano hypersurface is SRC, see
\cite{zhu}.

In this note, we consider a class of very special Fano hypersurfaces
over a field of positive characteristic. From now on, we fix $k$ to
be an algebraically closed field of characteristic $p>0$. Let
$X=X_{d,N}\subset \PP_k^N$ be the Fermat hypersurface defined by
$$
X_0^d+X_1^d+\cdots +X_N^d=0,
$$
where $d=p^r+1>3$ and $N\geq 3$. These hypersurfaces are special
since they are always unirational, see Th\'eor\`em 3.8 of
\cite{conduche}. Note that for fixed $N$, the hypersurface $X_{d,N}$
is of general type if $d>N$. Hence they are examples of varieties of
general type which are unirational. The variety $X$ is Fano if and
only if $d\leq N$. This condition is necessary for $X$ to have a
very free rational curve. In this case, we ask the following
\begin{ques}
If $d\leq N$, is there a very free rational curve on $X$?
\end{ques}
A positive answer for the cases when $N\geq 2p^r-1$ was given by
Corollaire 3.17 of \cite{conduche}. The following lemma shows that a
positive answer for large $N$ can always be deduced from a positive
answer for the case of a smaller $N$.
\begin{lem}\label{reduce dimension}
If $X_{d,N}$ contains a very free rational curve for $N=p^r+1$, then
$X_{d,N}$ contains a very free rational curve for all $N\geq p^r+1$.
\end{lem}
\noindent Hence we see that the most interesting case is when $d=N$.
Our first observation is
\begin{prop}\label{dimension jump}
Let $X=X_{d,d}$ be the Fermat hypersurface of degree $d=p^r+1$ in
$\PP_k^{d}$ and $M_e$ be the space of degree $e$ morphisms from
$\PP^1$ to $X$. Then for $M_e$ to have the expected dimension, $e$
has to be at least $p^r-1$. In particular, if $e<p^r-1$ then there
is no free rational curve of degree $e$.
\end{prop}
\noindent This is a special case of Th\'eor\`eme 3.16(a) of
\cite{conduche}. For very free rational curves on $X_{d,d}$ we have
the following
\begin{thm}\label{main theorem}
Let $X=X_{d,d}$ be the Fermat hypersurface of degree $d=p^r+1$ in
$\PP_k^{d}$. Let $f:C=\PP^1\rightarrow X$ be a rational curve of
degree $e$. If $mN<e\leq (m+1)(N-1)$ for some $0\leq m\leq N-3$,
then $f$ is not very free.
\end{thm}
\noindent The case when $m=0$ can be deduced from Th\'eor\`eme
3.16(a) of \cite{conduche}.
\begin{cor}
If there exists $N_0$ such that for all $e\geq N_0$ there is a very
free rational curve of degree $e$ on $X$, then $N_0> p^r(p^r-1)$.
\end{cor}

One of the simplest cases is when $X$ is the degree 5 Fermat
hypersurface in $\PP^5$ over $\overline{\mathbb{F}}_2$. The above
theorem shows that on this $X$ very free rational curves can only
exist in degrees $5,9,10$ and all $e\geq 13$. In the recent paper
\cite{bridges}, the authors showed that there is no very free
rational curve in degree 5 and they explicitly gave a very free
rational curve of degree $9$.

\begin{defn}
A \textit{rational normal curve of degree $e$} is a rational curve
$f:\PP^1\hookrightarrow\PP^N$ of degree $e$ whose linear span is a
subspace $\PP^e\subset\PP^N$.
\end{defn}
\noindent In particular, a necessary (and also sufficient) condition
to have a rational normal curve of degree $e$ in $\PP^N$ is that
$e\leq N$. One candidate of a very free rational curve of low degree
on $X$ is given by the following
\begin{prop}\label{rational normal}
Let $X=X_{d,N}$ be as above. If $C\subset X$ is a rational normal
curve of degree $N$ (viewed as a rational curve on $\PP^N$), then
$C$ is very free on $X$.
\end{prop}
\noindent This was first obtained in Th\'eor\`eme 3.16(b) of
\cite{conduche}. Proposition 3.19 of \cite{conduche} shows that such
rational normal curves exist only for $N\geq 2p^r-1$.

\section{Proofs}

Let $X=X_{d,N}$. We use the following diagram to investigate the
tangent sheaf of $X$.
\begin{equation}\label{basic diagram}
\xymatrix{
 &0 &0 & &\\
 0\ar[r] &T_X\ar[r]\ar[u] &T_{\PP^N}|_X\ar[r]\ar[u]
 &\calO_X(p^r+1)\ar[r] &0\\
 0\ar[r] &\mathscr{F}\ar[r]\ar[u]
 &\calO_X(1)^{N+1}\ar[u]\ar[r]^{(X_i^{p^r})}
 &\calO_X(p^r+1)\ar@{=}[u]\ar[r] &0\\
 &\calO_X\ar[u]\ar@{=}[r] &\calO_X\ar[u]^{(X_i)} & &\\
 &0\ar[u] &0\ar[u] & &
}
\end{equation}
We dualize the second column and get
$$
\xymatrix@C=0.5cm{
  0 \ar[r] & \Omega_{\PP^N}^1\otimes\calO_X(1) \ar[rr] && \calO_X^{N+1} \ar[rr]^{(X_i)} && \calO_X(1) \ar[r] & 0 }
$$
Let $F:X\rightarrow X$ be the Frobenius morphism. We apply $(F^*)^r$
to the above exact sequence and get
\begin{equation}
\xymatrix@C=0.5cm{
  0 \ar[r] & (F^*)^r\Omega_{\PP^N}^1\otimes\calO_X(p^r) \ar[rr] && \calO_X^{N+1} \ar[rr]^{(X_i^{p^r})} && \calO_X(p^r) \ar[r] & 0 }
\end{equation}
Compare this with the second row of \eqref{basic diagram}, we get
\begin{equation}\label{F expression}
\mathscr{F}\cong (F^*)^r\Omega_{\PP^N}^1\otimes\calO_X(p^r+1)
\end{equation}

\begin{proof}[Proof of Proposition \ref{rational normal}]
Since $f:C\rightarrow\PP^N$ is a rational normal curve, we have
$$
f^*\Omega_{\PP^N}^1\cong \calO_{\PP^1}(-N-1)^{\oplus N}
$$
Hence
$$
f^*\mathscr{F}=\calO_{\PP^1}((-N-1)p^r+(p^r+1)N)^{\oplus
N}=\calO_{\PP^1}(N-p^r)^N
$$
is very ample. The first column of diagram \eqref{basic diagram}
shows that $f^*T_X$ is very ample.
\end{proof}

\begin{proof}[Proof of Theorem \ref{main theorem}]
Consider the first column in diagram \eqref{basic diagram}, we know
that if the splitting of $f^*\mathscr{F}$ has a negative summand or
has at least two copies of $\calO_{\PP^1}$, then $f^*T_X$ is not
ample. To avoid this from happening, the best situation is when
$\Omega_{\PP^N}^1|_C$ is balanced. Now we assume that all the above
splittings are balanced. Let $a=[\frac{(N+1)e}{N}]=e+[\frac{e}{N}]$.
Then
$$
\Omega_{\PP^N}^1|_C\cong\calO(-a)^l\oplus\calO(-a-1)^{l'}
$$
with $l'=(N+1)e-Na$ and $l=N-l'$. Then it follows from \eqref{F
expression} that
$$
f^*\mathscr{F}\cong\calO(b_1)^l\oplus\calO(b_2)^{l'}
$$
with $b_1=-ap^r+e(p^r+1)$ and $b_2=-(a+1)p^r+e(p^r+1)$. Note that
$f^*\mathscr{F}$ is highly unbalanced unless $e$ is a multiple of
$N$. This is because when $e$ is a multiple of $N$, then $l'=0$ and
there is no summand of the form $\calO(b_2)$ in the above splitting.
If $mN<e<(m+1)N$, then $a=e+m$ and
$$
b_2=-(a+1)p^r+e(p^r+1)=e-(m+1)(N-1)
$$
Hence we have $b_2<0$ if $e<(m+1)(N-1)$. If $e=(m+1)(N-1)$, then
$$
l'=(N+1)e-Na=N-m-1
$$
The theorem follows easily from this computation.
\end{proof}

\begin{proof}[Proof of Proposition \ref{dimension jump}] A degree $e$
rational curve $f:\PP^1\rightarrow X$ can be written as
$$
t\mapsto[f_0:f_1:\cdots:f_d]
$$
where $f_i=\sum_{j=0}^{e}a_{ij}t^j$ are polynomials of degree at
most $e$. The condition for the image of $f$ to be contained in $X$
is given by $F=\sum f_i^d=0$ as a polynomial in $t$. If $f$ is free,
then Riemann-Roch tells us that $h^0(\PP^1,f^*T_X)=d+e-1$ which is,
by definition, the expected dimension of $M_e$. By explicit
computation we have
$$
f_i^d=(\sum_{j=0}^{e}a_{ij}t^j)^{p^r}(\sum_{j=0}^{e}a_{ij}t^j)=(\sum_{j=0}^{e}
a_{ij}^{p^r}t^{jp^r})(\sum_{j=0}^{e} a_{ij}t^j)
$$
If $e<p^r-1$ the coefficient of $t^{jp^r-1}$, $j=1,\ldots,e$, is
automatically 0. Hence the actual dimension of $M_e$ is bigger than
the expected dimension.
\end{proof}

\begin{proof}[Proof of Lemma \ref{reduce dimension}] It is easy to
realize $X_{d,N}$ as hyperplane section of $X_{d,N+1}$. This implies
that any very free rational curve on $X_{d,N}$ gives a very free
rational curve on $X_{d,N+1}$.

\end{proof}

\section{Acknowledgement}
This work was done when the author was a Ph.D. student at Columbia
University. The author would like to thank his advisor, Johan de
Jong, for many stimulating discussions. He would also like to thank
Yi Zhu for requesting a write-up of these results, which leads to
this note. Burt Totaro made many suggestions on a preliminary
version. Finally, many thanks to the referee for making the author
aware the work \cite{conduche} of D. Conduch\'e.

\end{document}